\documentclass[11pt]{amsart}
\usepackage{latexsym,amsfonts,amssymb}
\usepackage{amsmath}

\usepackage[top=30truemm,bottom=30truemm,left=25truemm,right=25truemm]{geometry}

\usepackage{amsthm}
\usepackage{indentfirst}
\usepackage{enumerate}
\usepackage{braket}
\usepackage{verbatim}
\usepackage{color}
\usepackage{colordvi,multicol}
\usepackage[marginal]{footmisc}
\IfFileExists{bbm.sty}{\usepackage{bbm}}{}
\usepackage[numbers]{natbib}
\usepackage[backref]{hyperref}
\hypersetup{
  pdftitle={Global L2 Approximation in the Bergman Spaces with Values in Holomorphic Vector Bundles},
  pdfauthor={Bo Xiao and Tianlong Yu}
}
\usepackage{lineno}

\usepackage{fancyhdr}

\allowdisplaybreaks[4]
\mathchardef\modelmathcomma=\mathcode`\,
\begingroup
\catcode`\,=\active
\gdef,{\modelmathcomma\allowbreak}
\endgroup
\AtBeginDocument{\mathcode`\,="8000\relax}

\makeatletter

\newcommand{\Rmnum}[1]{\expandafter\@slowromancap\romannumeral #1@}
\makeatother

\def \4l{\ \ \ \  }

\newtheorem{thm}{Theorem}
\newtheorem{theorem}[thm]{Theorem}

\numberwithin{equation}{section}
\numberwithin{thm}{section}
\date{}

\newtheorem{lemma}[thm]{Lemma}
\newtheorem{proposition}[thm]{Proposition}
\theoremstyle{definition}
\newtheorem{definition}{Definition}
\theoremstyle{plain}

\newtheorem*{thesisresult}{Theorem}
\newtheorem*{comparisonresult}{Proposition}
\theoremstyle{definition}
\newtheorem*{lempertmetric}{Definition}
\theoremstyle{plain}

\title{Global $L^2$ Approximation in the Bergman Spaces with Values in Holomorphic Vector Bundles}
\author[B. Xiao]{Bo Xiao\textsuperscript{2}}
\author[T. Yu]{Tianlong Yu\textsuperscript{1,*}}

\address[1]{Tianjin Renai College, Tianjin, China}
\email[Tianlong Yu]{yutianlong940309@163.com} 
\address[2]{School of Mathematical Sciences, University of Chinese Academy of Sciences, Beijing, China}
\email[Bo Xiao]{xb1994@amss.ac.cn}
\thanks{\textsuperscript{*}Corresponding author.}

\begin{document}

\begin{abstract}
In this paper, we get global $L^2$ approximation theorems in Bergman spaces with values in singular holomorphic vector bundles over Stein manifolds generalizing results of Li and Zhou. The main results combine strong openness property of multiplier submodule sheaf and twisted $L^2$ estimates of $\bar\partial$-equations.

\end{abstract}

\maketitle

\section{Introduction}
Let $D$ be a domain in $\mathbb{C}^n$, and $\phi$ a measurable function on $D$. Let
\begin{equation}
    A^2(D,\phi):=\{f\in\mathcal{O}(D)| \|f\|_{\phi}^2:=\int_D|f|^2e^{-\phi}d\lambda_n<\infty\}
\end{equation}
denotes the space of holomorphic functions on $D$ which are square integrable with respect to measure $e^{-\phi}d\lambda_n$, where $d\lambda_n$ is the Lebesgue measure on $\mathbb{C}^n$. Let $\phi_j$ be an increasing sequence of psh functions on $D$. Assume $\phi:=\lim_{j\to\infty}\phi_j\in psh(D)$, and $f\in A^2(D,\phi)$. Then strong openness property of psh functions implies that $|f|^2e^{-\phi_j}$ converges to $|f|^2e^{-\phi}$ in $L^1_{loc}$ sense.(\cite{GZ(2015b)}) Hence it is natural to ask whether we can get global convergence, i.e. $\lim_{j\to\infty}\||f|^2e^{-\phi_j}-|f|^2e^{-\phi}\|_{L^1(D)}=0$? Unfortunately, it is not always true. In fact, one may have strict inclusion: 
\begin{equation}
\cup_{j=1}^{\infty}A^2(D,\phi_j)\subsetneq A^2(D,\phi).    
\end{equation}
Acturally, in their joined paper\cite{FW(2018)}, Fornaess and Wu showed that for $D=\mathbb{C}^n$ and some toric psh functions $\phi_j$ increasing to $\phi$, one can have the above strict inclusion. Next natural question is that whether $\cup_{j=1}^{\infty}A^2(D,\phi_j)$ lies dense in $A^2(D,\phi)$? 

The study of this question has a long history. For the case of $D=\mathbb{C}^n$, in 1971, B.A.Taylor(\cite{Taylor(1971)}) proved the following
\begin{theorem}
Let $\phi_j$ be an increasing sequence of psh functions on $\mathbb{C}^n$. Let $\phi = \lim\limits_{j\rightarrow+\infty} \phi_j$. Assume $\int_Ke^{-\phi_1}d\lambda_n<\infty$ for every compact set K.  
Then the closure of $\cup_{j=1}^{\infty}A^2(\mathbb{C}^n,\phi_j+\log(1+\|z\|^2))$in the Hilbert space $A^2(\mathbb{C}^n,\phi+\log(1+\|z\|^2))$ contains $A^2(\mathbb{C}^n,\phi)$. 
\end{theorem}
Here $\|z\|^2 := |z_1|^2 +|z_2|^2 +\cdots + |z_n|^2$. After the strong openness property of psh functions be proved, Fornaess and Wu obtained the following theorem:
\begin{theorem}(\cite{FW(2018)} Theorem 1.4)\label{FW}
Let $\phi_j$ be an increasing sequence of psh functions on $\mathbb{C}^n$. Assume $\phi = \lim\limits_{j\rightarrow+\infty} \phi_j\in psh(\mathbb{C}^n)$. Then for any $\epsilon>0$, $\cup_{j=1}^{\infty}A^2(\mathbb{C}^n,\phi_j+\epsilon\log(1+\|z\|^2))$ lies dense in $A^2(\mathbb{C}^n,\phi+\epsilon\log(1+\|z\|^2))$.    
\end{theorem}

For a Holomorphic line bundle $L$ over a Stein manifold $X$ with a K\"{a}hler metric $\omega$ and a (singular) Hermitian metric $h$ on $L$, we denote $L^2(X,K_X\otimes L,h)$ the set of measurable sections $\alpha$ on $X$ of $K_X\otimes L$ such that
\begin{equation}
    \|\alpha\|^2_{h}:=\int_X\{\alpha,\alpha\}_h=\int_X|\alpha|^2_{\omega,h}dV_{\omega}<\infty
\end{equation}
And $A^2(X,K_X\otimes L,h)$ denotes the space of holomorphic elements in $L^2(X,K_X\otimes L,h)$. We call this space the Bergman space with value in a holomorphic line bundle.

In \cite{Li(2022)}, by using more advanced $L^2$ estimates established by Guan and Zhou (\cite{GZ(2015a)}), Li and Zhou generalized Theorem \ref{FW} in the above settings.
\begin{theorem}(\cite{Li(2022)} Theorem 1.4)\label{LiZhou1}
Let $(X,\omega)$ be a Stein manifold with a K\"{a}hler metric $\omega$, $\psi$ be a psh exhaustion function with a lower bound on $X$. Let $L$ be a holomorphic line bundle on $X$ and $\{h_k\}$, $h$ be a family of (singular) Hermitian metrics on $L$ such that

(1)$i\Theta_{(L,h_k)}\geq0$ and $i\Theta_{(L,h)}\geq0$ in the sense of currents.

(2)For any local holomorphic section $v$ of $K_X\otimes L, |v|_{\omega,h_k}$decreasingly converges to $|v|_{\omega,h}$.

Then  $\cup_{k=1}^{\infty}A^2(X,K_X\otimes L,h_ke^{-\psi})$is dense in $A^2(X,K_X\otimes L,he^{-\psi})$.
\end{theorem}
If we let $X$ be $\mathbb{C}^n$, $\omega$ be the standard K\"{a}hler metric on $\mathbb{C}^n$, and $L$ be the trivial line bundle with $h=e^{-\phi}$, $h_k=e^{-\phi_k}, \psi=\epsilon\log(1+\|z\|^2)$, then Theorem \ref{LiZhou1} recovers Theorem \ref{FW}.

Concerning the original question, we do not want to multiply the additional factor $e^{-\psi}$ to $h$ and $h_k$. If we restrict ourselves to more special $X$, then Li and Zhou obtained the following theorem without any twist factors.
\begin{theorem}(\cite{Li(2022)} Corollary 3.1)\label{LiZ2}
Let $(X,\omega)$ be a hyperconvex manifold with a K\"{a}hler metric $\omega$. Let $L$ be a holomorphic line bundle over $X$, and $\{h_j\}$, $h$ be a family of singular Hermitian metrics on $L$, such that

(1)$i\Theta_{(L,h_j)}\geq0$ and $i\Theta_{(L,h)}\geq0$ in the sense of currents.

(2)For any local holomorphic section $v$ of $K_X\otimes L, |v|_{\omega,h_j}$ decreasingly converges to $|v|_{\omega,h}$.

Then, $\cup_{j=1}^{\infty}A^2(X,K_X\otimes L,h_j)$is dense in $A^2(X,K_X\otimes L,h)$.
\end{theorem}
Recall that a complex manifold $X$ is called hyperconvex if there exists a smooth strictly psh function $\phi<0$ on $X$, such that $\phi^{-1}(-\infty,c)$ are relatively compact in $X$ for any negative $c$. The proof of Theorem \ref{LiZhou1} and Theorem \ref{LiZ2} relies on the strong openness property of psh functions and twisted Bochner-Kodaira-Nakano formula established by Guan and Zhou (\cite{GZ(2015a)}). 

An earlier vector bundle extension of the results of Li and Zhou was obtained by the first author in his doctoral thesis~\cite{XiaoThesis}. In that work, the metrics have semipositive Nakano curvature in the sense of Lempert~\cite{Lempert(2017)}: they admit increasing pointwise approximations by smooth Nakano semipositive Hermitian metrics. We recall the two approximation results from the thesis, without reproducing their proofs.

\begin{thesisresult}[{Xiao~\cite[Theorem~3.1 and Corollary~3.2]{XiaoThesis}}]
Let $(X,\omega)$ be a Stein manifold with a K\"ahler metric $\omega$, and let $\Psi$ be a psh exhaustion function with a lower bound on $X$. Let $E$ be a holomorphic vector bundle over $X$, and let $h$ and $\{h_k\}_{k\geq1}$ be Hermitian metrics on $E$ in the sense of Lempert. Assume that
\begin{enumerate}
\item $h$ is bounded below by a continuous Hermitian metric, and $h$ and every $h_k$ have semipostive Nakano curvature in the sense of Lempert;
\item for every holomorphic section $v$ of $K_X\otimes E$, $|v|_{\omega,h_k}$ decreases to $|v|_{\omega,h}$.
\end{enumerate}
Then
\[
\bigcup_{k=1}^{\infty}A^2(X,K_X\otimes E,h_ke^{-\Psi})
\quad\text{is dense in}\quad
A^2(X,K_X\otimes E,he^{-\Psi}).
\]
If $X$ is hyperconvex, the same assumptions on $h$ and $h_k$ imply that
\[
\bigcup_{k=1}^{\infty}A^2(X,K_X\otimes E,h_k)
\quad\text{is dense in}\quad
A^2(X,K_X\otimes E,h).
\]
\end{thesisresult}

On a Stein manifold, global holomorphic sections of $K_X\otimes E$ generate every fiber by Cartan's theorem A; hence the global-section convergence condition in the thesis is equivalent to the local-section condition used below. The proofs in the thesis use Lempert's strong openness theorem and the twisted $L^2$ estimates underlying the argument of Li and Zhou. The relation between Lempert's curvature condition and the optimal $L^2$ definition of singular Nakano semipositivity is recalled in Section~2. In particular, the results of Liu, Xiao, Yang and Zhou~\cite[Section~2.3]{Liu(2026)} imply that every metric with nonnegative Nakano curvature in the sense of Lempert is Nakano semipositive in the sense of Definition~\ref{sing Nak}. The main theorems below extend the approximation results of the thesis to this latter class of singular Hermitian metrics, without assuming an increasing smooth Nakano semipositive approximation.

Recently, strong openness properties has been established for any singular Nakano semipositive Hermitian vector bundles.(\cite{Liu(2026)}) Hence, it is reasonable to expect that the above density results also hold true if we replace $L$ with a higher rank vector bundle. The following are our main theorems:
\begin{theorem}\label{thm:main}
Let $(X,\omega)$ be a Stein manifold with a K\"ahler metric $\omega$, and let $\Psi$ be a psh exhaustion function with a lower bound on $X$. Let $E$ be a holomorphic vector bundle over $X$, and $\{h_k\}$, $h$ be a family of singular Hermitian metrics on $E$ such that

(1) $h_k$ and $h$ are Nakano semi-positive in the sense of Definition 1 in Section 2.

(2) For any local holomorphic section $v$ of $K_X\otimes E$, $|v|_{\omega,h_k}$ decreasingly converges to $|v|_{\omega,h}$.

Then, $\cup_{k=1}^{\infty}A^2(X,K_X\otimes E,h_ke^{-\Psi})$is dense in $A^2(X,K_X\otimes E,he^{-\Psi})$.
\end{theorem}

\begin{theorem}\label{main3}
Let $(X,\omega)$ be a hyperconvex manifold with a K\"{a}hler metric $\omega$. Let $E$ be a holomorphic vector bundle over $X$, $\{h_j\}$, $h$ be a family of singular Hermitian metrics on $E$, such that

(1)$h_j$ and $h$ are Nakano semi-positive in the sense of Definition 1 in Section 2..

(2)For any local holomorphic section $v$ of $K_X\otimes E, |v|_{\omega,h_j}$decreasingly converges to $|v|_{\omega,h}$.

Then, $\cup_{j=1}^{\infty}A^2(X,K_X\otimes E,h_j)$is dense in $A^2(X,K_X\otimes E,h)$.    
\end{theorem}
Here $A^2(X,K_X\otimes E,h)$ is defined exactly in the same way as in the line bundle case. And for the notions of singular Hermitian metrics and its positivity concepts on holomorphic vector bundles, we shall recall in the second section of this paper. When rank$E=1$, Theorem \ref{thm:main} implies Theorem \ref{LiZhou1}, and Theorem \ref{main3} implies Theorem \ref{LiZ2} (See Section 2).

In order to prove Theorem \ref{thm:main} and Theorem \ref{main3}, since we do not have twisted Bochner-Kodaira-Nakano formula directly in the singular Hermitian vector bundle case, we take a different approach as in \cite{Li(2022)}. For Theorem \ref{thm:main}, by using the definition of Nakano semi-positivity of singular Hermitian metrics, we solve a $\bar\partial$-equation with a specifically chosen weight. For Theorem \ref{main3}, We generalize Berndtsson's twisted $L^2$ estimates for trivial line bundle cases(\cite{Berndtsson(2001)}) which is also used in \cite{FW(2018)} to higher rank vector bundle cases. And both theorems use strong openness properties for Hermitian vector bundles to complete the proof.

\section{Preliminaries}
In this section, we first recall two lemmas in Hodge theory and basic $L^2$ theory of $\bar\partial$-equations.

\begin{lemma}(\cite{GZ(2015a)} Lemma 4.2)\label{l:positve}
Let $X$ be  a K\"{a}hler manifold of dimension $n$ with
a K\"{a}hler metric $\omega$. Let $(E,h)$ be a hermitian holomorphic vector bundle over $X$. $\theta$ be a continuous $(1,0)$ form on $X$. Then we have
$$[\sqrt{-1}\theta\wedge\bar\theta,\Lambda_{\omega}]\alpha=\bar\theta\wedge(\alpha\llcorner(\bar\theta)^\sharp\big),$$
for any $(n,1)$ form $\alpha$ with value in $E$. Moveover, for any positive $(1,1)$ form $\beta$, we have $[\beta,\Lambda_{\omega}]$ is semipositive.    
\end{lemma}

\begin{lemma}(\cite{Demailly(1982)} Lemma 6.9)\label{ext of dbar}
 Let $\Omega$ be an open subset of $\mathbb{C}^n$ and
$Z$ be a complex analytic subset of $\Omega$. Assume that $u$ is a $(p,q-1)$-form with $L^2_{loc}$ coefficients and $v$ is a $(p,q)$-form with $L^1_{loc}$ coefficients such that $\bar\partial u = v$ on $\Omega-Z$ (in the sense of distribution theory). Then $\bar\partial u = v$ on $\Omega$    
\end{lemma}

We also need the following approximation result of psh functions on Stein manifolds.
\begin{lemma}(\cite{FN(1980)} Theorem 5.5)\label{psh-approxi}
Let X be a Stein manifold, $\varphi$ be a psh function on X. Then there exists a sequence of smooth strictly psh functions $\{\varphi_j\}_{j\in\mathbb{N}}$ on X decreasingly converges to $\varphi$.
\end{lemma}
Next, let us recall the notion of singular Hermitian metrics and its positivity concepts on holomorphic vector bundles. 

Let $X$ be a complex manifold, $E$ be a holomorphic vector bundle of rank $r$ over $X$. A singular Hermitian metric on $E$ is a measurable map $h$ from the base manifold $X$ to the set of Hermitian forms on the fibers of $E$, which assigns to almost every point $x \in X$ a Hermitian form $h(x)$ on $E_x$ satisfying $0 < \det h(x) < +\infty$.(i.e. for almost every $x\in X$, $h(x)$ is an ordinary positive definite matrix)(\cite{PT(2018)})

For comparison with the earlier approximation results, we also recall the metric and curvature notions introduced by Lempert.

\begin{lempertmetric}[{Lempert~\cite{Lempert(2017)}}]
Let $E\to X$ be a holomorphic vector bundle of finite rank. A Hermitian metric in the sense of Lempert is a measurable squared-norm function $h:E\to[0,+\infty]$, with $0<\det h<+\infty$ almost everywhere, for which there is an increasing sequence of $C^2$-smooth Hermitian metrics $g_\nu$ on $E$ such that
\[
|e|_{g_\nu}^2\nearrow |e|_h^2
\qquad\text{for every }e\in E.
\]
The Chern curvature of $h$ is said to dominate $0$ in the sense of Nakano if the approximating metrics can be chosen to be $C^2$-smooth and Nakano semipositive in the usual sense. Thus the approximating metrics are defined on $E$ over the whole base $X$.
\end{lempertmetric}

It is well-known that, when $h$ is smooth, $h$ is Griffiths semi-positive if and only if $\log|u|^2_{h^*}$ is plurisubharmonic for any local holomorphic section $u$ of the dual bundle.(\cite{LiuXu(2024)}) This characterization naturally gives a definition of Griffiths semi-positivity for singular Hermitian metrics. Subsequently, based on Deng-Ning-Wang-Zhou's characterization of Nakano positivity of $C^2$ smooth Hermitian metrics(\cite{DNWZ(2021)}), Inayama(\cite{Inayama(2022)}) made the following definition of Nakano positivity of singular Hermitian metrics.
\begin{definition}(\cite{Liu(2026)} Definition 1.2)\label{sing Nak}
Let $h$ be a singular Hermitian metric on a holomorphic vector bundle $E$ over a complex manifold $X$.  $h$ is said to be Nakano semi-positive if $h$ is Griffiths semi-positive and is $L^2$ optimal, that is,
	for any Stein open subset $U$ such that $E|_U$ is trivial, any K\"ahler form $\omega_U$ on $U$, any $\psi \in {\rm Spsh}(U)\cap C^\infty(U)$ and any $\overline{\partial}$-closed $f \in C^\infty_{(n,1)} (U,E|_U)$, there exists  $u\in  L^2_{(n,0)} (U,\omega_U,E|_U,he^{-\psi})$ such that $\overline{\partial}u = f$ and
\begin{equation}\label{def Hormander estimate}
	\int_{U}^{} |u|^2_{\omega_U,h} e^{-\psi} dV_{\omega_U}
	\le \int_{U}^{} \langle B^{-1}_{\omega_U,\psi} f,f \rangle_{\omega_U,h} e^{-\psi} dV_{\omega_U},
	\end{equation}
	provided that the right hand side is finite, where $B_{\omega_U,\psi}=[i\partial\overline{\partial}\psi \otimes {\rm Id}_E, \Lambda_{\omega_U}]$.
\end{definition}
When rank$E=1$, $h$ is Nakano semi-positive in the sense of Definition \ref{sing Nak} is equivalent to $i\Theta_{(E,h)}\geq0$ in the sense of currents (See for example \cite{Demailly(2010)}). Hence Theorem \ref{thm:main} implies Theorem \ref{LiZhou1}, and Theorem \ref{main3} implies Theorem \ref{LiZ2}.

The implication needed to compare the thesis results with the present theorems follows from the preservation of singular Nakano semipositivity under increasing limits proved in~\cite[Section~2.3]{Liu(2026)}.

\begin{comparisonresult}[{Liu--Xiao--Yang--Zhou~\cite[Section~2.3]{Liu(2026)}}]
Let $h$ be a singular Hermitian metric on a holomorphic vector bundle $E\to X$, and let $g_\nu$ be singular Hermitian metrics on $E$ that are Nakano semipositive in the sense of Definition~\ref{sing Nak}. If $g_\nu$ increases pointwise to $h$, then $h$ is also Nakano semipositive in the sense of Definition~\ref{sing Nak}.
In particular, a metric whose Chern curvature dominates $0$ in the sense of Lempert is Griffiths semipositive and satisfies the optimal $L^2$ estimate condition of Definition~\ref{sing Nak}.
\end{comparisonresult}

The special case of increasing limits of smooth Nakano semipositive metrics is also discussed by Inayama~\cite{Inayama(2022)}. This implication shows that the hypotheses of the approximation results in~\cite{XiaoThesis} are covered by the positivity framework used here.

Recently, strong openness property has been established for multiplier submodule sheaf of singular Nakano semi-positive Hermitian vector bundles. The following is one of its statements.
\begin{theorem}(\cite{Liu(2026)} Corollary 5.4)\label{strong open v.b.}
Let $E$ be a trivial holomorphic vector bundle over a domain $D\subset\mathbb{C}^n$. $h$, $\{h_j\}_{j\in\mathbb{N}}$ be a family of singular Nakano semi-positive Hermitian metrics on $E$ such that $h_j\geq h$ and $\varphi_j:=-\log\det h_j$ converges to $\varphi:=-\log\det h$ locally in measure. Let $f$ be a holomorphic section of $E$ on $D$. Then $|f|^2_{h_j}$ converges to $|f|^2_{h}$ in $L^1_{loc}$ sense.     
\end{theorem}

\section{Proof of Theorem~\ref{thm:main}}

Fix $\chi\in C^\infty(\mathbb R)$ such that $0\leq\chi\leq1$, $\chi=1$ on $(-\infty,1]$, $\chi=0$ on $[2,+\infty)$, and $|\chi'|\leq2$.

\begin{lemma}\label{lem:cutoff}
Let $(X,\omega)$ be a Stein manifold with a K\"{a}hler metric $\omega$. Let $\Psi>0$ be a smooth strictly plurisubharmonic exhaustion function, and let $h$ be a Nakano semipositive singular Hermitian metric on $E$. Suppose that $f$ is a holomorphic section of $K_X\otimes E$ and that, for some $c\geq1$,
\begin{equation}\label{finite}
 \int_{\{\Psi<e^{2c}-1\}}|f|_{\omega,h}^2e^{-\Psi}dV_\omega<+\infty.
\end{equation}
Then the equation
\begin{equation}\label{eq:cutoff-equation}
 \bar\partial u
 =f\wedge\bar\partial\left(\chi\left(\frac{\log(1+\Psi)}{c}\right)\right)
\end{equation}
has a solution $u_c$ satisfying
\begin{equation}\label{eq:cutoff-bound}
 \int_X|u_c|_{\omega,h}^2e^{-\Psi}dV_\omega
 \leq\frac{4e^2}{c}
 \int_{\{\Psi<e^{2c}-1\}}|f|_{\omega,h}^2e^{-\Psi}dV_\omega.
\end{equation}
\end{lemma}

\begin{proof}
Put
\[
 \phi_c=\Psi-\frac1c\log(1+\Psi).
\]
Since $c\geq1$ and $\Psi>0$, we have
\begin{equation}\label{eq:scalar-curvature}
\begin{aligned}
 i\partial\bar\partial\phi_c
 &=\left(1-\frac1{c(1+\Psi)}\right)i\partial\bar\partial\Psi
   +\frac1{c(1+\Psi)^2}i\partial\Psi\wedge\bar\partial\Psi\\
 &\geq\frac1{c(1+\Psi)^2}i\partial\Psi\wedge\bar\partial\Psi.
\end{aligned}
\end{equation}
In particular, $\phi_c$ is strictly plurisubharmonic.
Next, we need a pointwise calculation of $\langle[i\partial \Psi \wedge \overline{\partial} \Psi,\Lambda_{\omega}]^{-1}(f\wedge\bar\partial\Psi),f\wedge\bar\partial\Psi\rangle _{\omega,h}$.

At a given point $x_0\in X_j$, we can choose a basis $\{dz_1,\cdots,dz_n\}$ of $T_{x_0}^{*}X$ such that
$$\omega=i\sum_{j=1}^ndz_j\wedge d\bar{z_j},~~~~~~\partial\Psi=bdz_1$$
Let $\{e_1,\cdots,e_r\}$ be an orthonormal frame of $(E_{x_0},h_{x_0})$, $f=\sum_{l=1}^rc_ldz_1\wedge\cdots\wedge dz_n\otimes e_l$. Then $i\partial\Psi\wedge\bar\partial\Psi=i|b|^2dz_1\wedge d\bar{z_1}$ and $f\wedge\bar\partial\Psi=\sum_{l=1}^rc_l\bar{b}dz_1\wedge\cdots\wedge dz_n\wedge d\bar{z_1}\otimes e_l$. By [\cite{Demailly(2010)} Chapter VI, Proposition 5.8], we have
$$[i\partial \Psi \wedge \overline{\partial} \Psi,\Lambda_{\omega}]^{-1}(f\wedge\bar\partial\Psi)=\sum_{l=1}^rc_lb^{-1}dz_1\wedge\cdots\wedge dz_n\wedge d\bar{z_1}\otimes e_l$$
Hence we have 
\begin{equation}\label{pointwise calculation}
\begin{split}
\langle[i\partial \Psi \wedge \overline{\partial} \Psi,\Lambda_{\omega}]^{-1}(f \wedge\bar\partial\Psi),f \wedge\bar\partial\Psi\rangle _{\omega,h}= \sum_{l=1}^{r}|c_l|^2 
\\=|f|^2_{\omega,h}
\end{split} 
\end{equation}
We set
\[
 B=B_{\omega,\phi_c},\qquad
 \lambda=\chi'\left(\frac{\log(1+\Psi)}c\right)
          \frac1{c(1+\Psi)}f\wedge\bar\partial\Psi,
\]
Then by (\ref{pointwise calculation}) and Lemma \ref{l:positve}
\[
 \langle B^{-1}\lambda,\lambda\rangle_{\omega,h}
 \leq\frac1c\left|\chi'\left(\frac{\log(1+\Psi)}c\right)\right|^2|f|_{\omega,h}^2.
\]
The form $\lambda$ is smooth and $\bar\partial$-closed. Its support is contained in $\{e^c-1\leq\Psi\leq e^{2c}-1\}$, where
\[
 e^{-\phi_c}=e^{-\Psi}(1+\Psi)^{1/c}\leq e^2e^{-\Psi}.
\]
Since $X$ is Stein, we can choose $r$=rank$E$ holomorphic sections of $E$ on $X_j$ such that they generate the fiber of $E$ at some point of $X$. Then these sections generate the fiber of $E$ outside an analytic subset of $X$ and hence we can choose a Zariski open Stein subdomain $Y$ of $X$ such that $E$ is trivial on $Y$. Therefore, by Definition \ref{sing Nak}, there exists $u\in L^2_{(n,0)} (Y,\omega|_{Y},E|_{Y},he^{-\phi_c})$ such that $\bar\partial u=\lambda$ on $Y$. Since $h$ is Griffiths semi-positive, by [\cite{PT(2018)} Lemma 2.2.4], the coefficients of $u$ are in $L^2_{loc}$. Hence we have $\bar\partial u=v$ on $X$ by Lemma \ref{ext of dbar}, and
$$\int_{X}|u|^2_{\omega,h}e^{-\phi_c}dV_{\omega}\leq\int_{X}^{} \langle B^{-1}\lambda,\lambda \rangle_{\omega,h} e^{-\phi_c} dV_{\omega}\leq\frac{4e^2}{c}
 \int_{\{\Psi<e^{2c}-1\}}|f|_{\omega,h}^2e^{-\Psi}dV_\omega$$
Since $e^{-\phi_c}\geq e^{-\Psi}$, this proves \eqref{eq:cutoff-bound}.
\end{proof}

\begin{proof}[Proof of Theorem~\ref{thm:main}]
Adding a constant to $\Psi$, we may assume that $\Psi>0$. Let
\[
 f\in A^2(X,K_X\otimes E,he^{-\Psi}),\qquad
 M:=\int_X|f|_{\omega,h}^2e^{-\Psi}dV_\omega.
\]
Fix $c\geq1$ and set
\[
 D_c=\{\Psi<e^{2c}-1\}.
\]
Its closure is compact. Since $\Psi$ is locally bounded, we have
$$\int_{\overline{D_c}}|f|_{\omega,h}^2dV_\omega<\infty.$$
By the assumption (2) of Theorem\ref{main3}, we know that $h_j\geq h$ and $\det h_j$ converges to $\det h$ almost everywhere. Hence by Theorem\ref{strong open v.b.}, there exists an integer $N_c$ such that
\[
 \int_{D_c}|f|_{\omega,h_j}^2e^{-\Psi}dV_\omega<+\infty
 \qquad(j\geq N_c).
\]
And by dominated convergence theorem,
\[
 \int_{D_c}|f|_{\omega,h_j}^2e^{-\Psi}dV_\omega
 \longrightarrow\int_{D_c}|f|_{\omega,h}^2e^{-\Psi}dV_\omega.
\]
as $j$ tends to $\infty$. Choose $j_c\geq N_c$ so that
\begin{equation}\label{eq:compact-integral}
 \int_{D_c}|f|_{\omega, h_{j_c}}^2e^{-\Psi}dV_\omega
 \leq2\int_{D_c}|f|_{\omega,h}^2e^{-\Psi}dV_\omega\leq2M.
\end{equation}

By Lemma~\ref{psh-approxi}, there are smooth strictly plurisubharmonic functions $\Psi_k$ on $X$ decreasing to $\Psi$. Since $\Psi_k\geq\Psi>0$, each $\Psi_k$ is also an exhaustion function. Apply Lemma~\ref{lem:cutoff} to $h_{j_c}$ and $\Psi_k$. It gives a solution $u_{c,j_c,k}$ satisfying
$$\bar\partial u_{c,j_c,k}
 =f\wedge\bar\partial\left(\chi\left(\frac{\log(1+\Psi_k)}{c}\right)\right).$$
 with estimates
\begin{equation}\label{eq:singular-solutions}
\begin{aligned}
 \int_X|u_{c,j_c,k}|_{\omega,h_{j_c}}^2e^{-\Psi_k}dV_\omega
 &\leq\frac{4e^2}{c}
 \int_{\{\Psi_k<e^{2c}-1\}}|f|_{\omega,h_{j_c}}^2e^{-\Psi_k}dV_\omega\\
 &\leq\frac{4e^2}{c}\int_{D_c}|f|_{\omega,h_{j_c}}^2e^{-\Psi}dV_\omega
 \leq\frac{8e^2}{c}M.
\end{aligned}
\end{equation}
Here the metric $h_{j_c}$ is fixed. Set
\begin{equation}\label{eq:holomorphic-approximants}
 F_{c,j_c,k}
 =\chi\left(\frac{\log(1+\Psi_k)}c\right)f-u_{c,j_c,k}.
\end{equation}
Then $F_{c,j_c,k}$ is a holomorphic section of $K_X\otimes E$. Moreover, \eqref{eq:compact-integral} and \eqref{eq:singular-solutions} imply
\[
 \int_X|F_{c,j_c,k}|_{\omega,h_{j_c}}^2e^{-\Psi_k}dV_\omega
 \leq\left(4+\frac{16e^2}{c}\right)M.
\]
Since $h_{j_c}\geq h$,$\Psi_k\leq\Psi_1$ and $h$ is Griffith semi-positive, by [\cite{PT(2018)} Lemma 2.2.4], we have uniform ordinary $L^2$ bounds of $F_{c,j_c,k}$ on compact subsets with respect to $k$. Hence there exists a subsequence of $F_{c,j_c,k}$ converging uniformly on compact subsets to a holomorphic section $F_{c,j_c}$. Fatou's lemma and \eqref{eq:singular-solutions} give
\begin{equation}\label{eq:limit-estimate}
 \int_X\left|F_{c,j_c}-\chi\left(\frac{\log(1+\Psi)}c\right)f\right|_{\omega, h_{j_c}}^2
 e^{-\Psi}dV_\omega
 \leq\frac{8e^2}{c}M.
\end{equation}
By \eqref{eq:compact-integral}, the cutoff section is square integrable for $h_{j_c}e^{-\Psi}$. Hence
\[
 F_{c,j_c}\in A^2(X,K_X\otimes E,h_{j_c}e^{-\Psi}).
\]
Finally, $h\leq h_{j_c}$ and \eqref{eq:limit-estimate} imply
\begin{equation}\label{eq:main-error}
\begin{aligned}
 \int_X|F_{c,j_c}-f|_{\omega,h}^2e^{-\Psi}dV_\omega
 &\leq\frac{16e^2}{c}M
   +2\int_{\{\Psi>e^c-1\}}|f|_{\omega,h}^2e^{-\Psi}dV_\omega\\
 &\longrightarrow0\qquad(c\to+\infty).
\end{aligned}
\end{equation}
This proves the theorem.
\end{proof}

\section{Proof of Theorem \ref{main3}}
In this section, we prove Theorem \ref{main3} following the idea of \cite{FW(2018)}. Let us recall the following weighted $\bar\partial$ estimate established by Berndtsson.

\begin{theorem} [\cite{Berndtsson(2001)}] \label{th:Hormander1}
Let $\Omega \subset \mathbb{C}^n$ be a pseudoconvex domain and $\varphi \in psh(\Omega)$.  $\psi$ is a $C^2$ real function satisfying
$$
r i\partial \overline{\partial} (\varphi+\psi) \geq i\partial \psi \wedge \overline{\partial} \psi
$$
in the sense of distributions for some $0<r<1$. Then for each $\overline{\partial}$-closed $(0,1)$-form $v$, there is a  solution of $\overline{\partial} u =v$ which satisfies
$$
\int _{\Omega} |u|^2e^{\psi-\varphi}d\lambda \leq \frac{1}{(1-\sqrt{r})^2} \int _{\Omega} |v|_\Theta^2e^{\psi-\varphi}d\lambda
$$
for every continuous positive $(1,1)-$ form $\Theta$ with $i\partial \overline{\partial}(\varphi+\psi)\ge \Theta$ in the sense of distributions.
\end{theorem}

In order to prove Theorem \ref{main3}, we first generalize Berndtsson's result to vector bundle valued forms with possibly singular Hermitian metrics on Stein manifolds. And this proposition seems useful in its own right. 
\begin{proposition}\label{Berndtsson}
Let $(X,\omega)$ be a $n$-dimensional Stein manifold with a K\"{a}hler metric $\omega$. $\varphi \in psh(X)$.  $\psi$ is a $C^{\infty}$-real function satisfying
$$
r i\partial \overline{\partial} (\varphi+\psi) \geq i\partial \psi \wedge \overline{\partial} \psi
$$
in the sense of distributions for some $0<r<1$. $E$ be a Hermitian vector bundle over $X$ with singular metric $h$ which is Nakano semi-positive. Then for each smooth $\overline{\partial}$-closed $(n,1)$-form $v$ with value in E, there is a  solution of $\overline{\partial} u =v$ which satisfies
\begin{eqnarray}\label{ineq:estimate1}
\int _{X} |u|_{\omega,h}^2e^{\psi-\varphi}dV_{\omega} \leq \frac{1}{(1-\sqrt{r})^2} \int _{X}\langle[\Theta,\Lambda_{\omega}]^{-1}v,v\rangle _{\omega,h}e^{\psi-\varphi}dV_{\omega}
\end{eqnarray}
for every continuous positive $(1,1)-$ form $\Theta$ with $i\partial \overline{\partial}(\varphi+\psi)\ge \Theta$ in the sense of distributions.
\end{proposition}
\begin{proof}
Since $X$ is Stein, we may assume the right hand side of (\ref{ineq:estimate1}) is finite.

We choose an increasing sequence of Stein domains $X_j$ in $X$ such that $\cup_{j=1}^{\infty}X_j=X$ and $X_j\subset\subset X_{j+1}$.

Since $X$ is Stein, we may assume $X$ is a closed complex submanifold of $\mathbb{C}^n$ and $U$ is a Stein open neighborhood of $X$ in $\mathbb{C}^n$. Let $\pi: U\to X$ be a holomorphic retraction map and $l: X\to U$ be the inclusion. Then 
$$
r i\partial \overline{\partial} (\varphi\circ\pi+\psi\circ\pi) \geq i\partial (\psi\circ\pi) \wedge \overline{\partial} (\psi\circ\pi)
$$
and
$$i\partial \overline{\partial}(\varphi\circ\pi+\psi\circ\pi)\ge \pi^{*}\Theta$$ 
hold on $U$.

We can choose Stein domains $U_j$ in $U$ such that $\cup_{j=1}^{\infty}U_j=U$ and $X_j\subset\subset U_j\subset\subset U_{j+1}$.

As in the proof of Theorem 2.3 in \cite{FW(2018)}             , for each $j$, we take $\tilde{\varphi_j}=(\varphi\circ\pi)*\theta_{\epsilon_j}+\frac{1}{j}\|z\|^2$ on $U_{j+1}$ with $\theta$ be the standard smoothing kernel and $\epsilon_j\to 0$ rapidly as $j\to\infty$. Then we have the decreasing limit $\lim_{j\to\infty}\tilde{\varphi_j}=\varphi\circ\pi$ on $U$, amd 
$$
r i\partial \overline{\partial} (\tilde{\varphi_j}+\psi\circ\pi) \geq i\partial (\psi\circ\pi) \wedge \overline{\partial} (\psi\circ\pi)
$$
and
$$i\partial \overline{\partial}(\tilde{\varphi_j}+\psi\circ\pi)\ge \pi^{*}\Theta$$ 
hold on $U_j$.

Then $\{\varphi_j:=\tilde{\varphi_j}\circ l\}$ is a decreasing sequence of smooth strictly psh functions on $X$ such that $\lim_{j\to\infty}\varphi_j=\varphi$, and with
\begin{equation}\label{34}
r i\partial \overline{\partial} (\varphi_j+\psi) \geq i\partial \psi \wedge \overline{\partial} \psi
\end{equation}
and
\begin{equation}\label{35}
i\partial \overline{\partial}(\varphi_j+\psi)\geq \Theta
\end{equation}
holding on $X_j$.

By Lemma \ref{l:positve}, for any positive $(1,1)$ form $\beta$, we have $[\beta,\Lambda_{\omega}]$ is a positive operator on $\Lambda^{n,1}T^{*}_{x_0}X$. Thus by the inequality (\ref{35}), we have
$$\langle[\Theta,\Lambda_{\omega}]^{-1}v,v\rangle _{\omega,h}\geq\langle[i\partial \overline{\partial}(\varphi_j+\psi),\Lambda_{\omega}]^{-1}v,v\rangle _{\omega,h}$$
Since $\psi$ is bounded on $X_j$, by our assumption, we have $$\int_{X_j}\langle[i\partial \overline{\partial}(\varphi_j+\psi),\Lambda_{\omega}]^{-1}v,v\rangle _{\omega,h}e^{-\varphi_j-\psi}dV_{\omega}<\infty$$

Therefore by Definition \ref{sing Nak}, as in the proof of Lemma \ref{lem:cutoff}, we can find a solution $u_j$ of the equation $\bar\partial u=v$ on $X_j$ with estimate:
$$\int_{X_j}|u_j|^2_{\omega,h}e^{-\varphi_j-\psi}dV_{\omega}<\infty$$
We may assume $u_j$ is chosen to be the unique solution that is perpendicular to all holomorphic (n,0)-forms in $L^2_{(n,0)} (X_j,\omega|_{X_j},E|_{X_j},he^{-\varphi_j})$.

Because of the boundedness of $\psi$ on $X_j$, we have 
$$L^2_{(n,0)} (X_j,\omega|_{X_j},E|_{X_j},he^{-\varphi_j-\psi})=L^2_{(n,0)} (X_j,\omega|_{X_j},E|_{X_j},he^{-\varphi_j})$$
and $u_je^{\psi}\in L^2_{(n,0)} (X_j,\omega|_{X_j},E|_{X_j},he^{-\varphi_j-\psi})$. Furthermore we have $u_je^{\psi}\bot \ker\bar\partial$ in $L^2_{(n,0)} (X_j,\omega|_{X_j},E|_{X_j},he^{-\varphi_j-\psi})$. i.e.
$$\int_{X_j}\langle u_je^{\psi},g\rangle_{\omega,h}e^{-\varphi_j-\psi}dV_{\omega}=\int_{X_j}\langle u_j,g\rangle_{\omega,h}e^{-\varphi_j}dV_{\omega}=0$$
for all holomorphic (n,0)-forms in $L^2_{(n,0)} (X_j,\omega|_{X_j},E|_{X_j},he^{-\varphi_j-\psi})$.

Thus $u_je^{\psi}$ is the $L^2_{(n,0)} (X_j,\omega|_{X_j},E|_{X_j},he^{-\varphi_j-\psi})$-minimal solution of the equation
$$\bar\partial\tilde{u}=\bar\partial(u_je^{\psi})$$
Since $\bar\partial$ is hypoelliptic in bidegree $(n,0)$[\cite{Demailly(2010)} Chapter I, Corollary 3.30], we have $u_j$ and hence $\bar\partial(u_je^{\psi})$ is smooth on $X_j$. Therefore, by Definition \ref{sing Nak}, we have
\begin{equation}\label{36}
\begin{split}
\int_{X_j}|u_j|^2_{\omega,h}e^{\psi-\varphi_j}dV_{\omega}=&\int_{X_j}|u_je^{\psi}|^2_{\omega,h}e^{-\varphi_j-\psi}dV_{\omega}
\\\leq&\int_{X_j}\langle[i\partial \overline{\partial}(\varphi_j+\psi),\Lambda_{\omega}]^{-1}\bar\partial(u_je^{\psi}),\bar\partial(u_je^{\psi})\rangle _{\omega,h}e^{-\varphi_j-\psi}dV_{\omega}
\\=&\int_{X_j}\langle[i\partial \overline{\partial}(\varphi_j+\psi),\Lambda_{\omega}]^{-1}(v+u_j\wedge\bar\partial\psi),v+u_j\wedge\bar\partial\psi\rangle _{\omega,h}e^{\psi-\varphi_j}dV_{\omega}
\\\leq&(1+\frac{1}{t})\int_{X_j}\langle[i\partial \overline{\partial}(\varphi_j+\psi),\Lambda_{\omega}]^{-1}v,v\rangle _{\omega,h}e^{\psi-\varphi_j}dV_{\omega}
\\&+(1+t)\int_{X_j}\langle[i\partial \overline{\partial}(\varphi_j+\psi),\Lambda_{\omega}]^{-1}(u_j\wedge\bar\partial\psi),u_j\wedge\bar\partial\psi\rangle _{\omega,h}e^{\psi-\varphi_j}dV_{\omega}
\\\leq&(1+\frac{1}{t})\int_{X_j}\langle[\Theta,\Lambda_{\omega}]^{-1}v,v\rangle _{\omega,h}e^{\psi-\varphi_j}dV_{\omega}
\\&+(1+t)\int_{X_j}\langle[\frac{1}{r}i\partial \psi \wedge \overline{\partial} \psi,\Lambda_{\omega}]^{-1}(u_j\wedge\bar\partial\psi),u_j\wedge\bar\partial\psi\rangle _{\omega,h}e^{\psi-\varphi_j}dV_{\omega}
\end{split}
\end{equation}
where $0<t<1$, and the last inequality is a consequence of the inequalities (\ref{34}) and (\ref{35}).

By a similar calculation as in the proof of Lemma \ref{lem:cutoff}, we have 
\begin{equation}\label{37}
\langle[i\partial \psi \wedge \overline{\partial} \psi,\Lambda_{\omega}]^{-1}(u_j\wedge\bar\partial\psi),u_j\wedge\bar\partial\psi\rangle _{\omega,h}=|u_j|^2_{\omega,h}  
\end{equation}
Because $\psi$ is bounded on $X_j$, we have
$$\int_{X_j}|u_j|^2_{\omega,h}e^{\psi-\varphi_j}dV_{\omega}\leq C_{\psi}\int_{X_j}|u_j|^2_{\omega,h}e^{-\varphi_j-\psi}dV_{\omega}<\infty$$
Hence by (\ref{36}) and (\ref{37}), we get
$$\int_{X_j}|u_j|^2_{\omega,h}e^{\psi-\varphi_j}dV_{\omega}\leq\frac{1+\frac{1}{t}}{1-(1+t)r}\int_{X_j}\langle[\Theta,\Lambda_{\omega}]^{-1}v,v\rangle _{\omega,h}e^{\psi-\varphi_j}dV_{\omega}$$
We take $t=\frac{1}{\sqrt{r}}-1$ to make the above coefficient get the minimum $\frac{1}{(1-\sqrt{r})^2}$.

Therefore for each $j$, we have 
$$\int_{X_j}|u_j|^2_{\omega,h}e^{\psi-\varphi_j}dV_{\omega}\leq\frac{1}{(1-\sqrt{r})^2}\int_{X}\langle[\Theta,\Lambda_{\omega}]^{-1}v,v\rangle _{\omega,h}e^{\psi-\varphi}dV_{\omega}<\infty$$
Hence by a diagonal argument, we may assume $\{u_j\}$ weakly converges to $u$ in any Hilbert spaces $L^2_{(n,0)} (X_k,\omega|_{X_k},E|_{X_k},he^{\psi-\varphi_k})$ and also in $L^2_{loc}$ space. Hence we have
$$\bar\partial u=v$$
on $X$. And for each fixed $k$, we have
\begin{equation}
\begin{split}
\int_{X_k} |u|_{\omega,h}^2 e^{\psi-\varphi_k} dV_{\omega}
\leq&\liminf \limits_{j \rightarrow + \infty}\int_{X_k}  |u_j|_{\omega,h}^2 e^{\psi - \varphi_k}dV_{\omega} 
\\\leq&\liminf \limits_{j \rightarrow + \infty}\int_{X_j}  |u_j|_{\omega,h}^2 e^{\psi - \varphi_j}dV_{\omega} 
\\\leq&\frac{1}{(1-\sqrt{r})^2}\int_X \langle[\Theta,\Lambda_{\omega}]^{-1}v,v\rangle _{\omega,h}e^{\psi - \varphi}  dV_{\omega}
\end{split}
\end{equation}
Hence we have 
\begin{equation}
\begin{split}
\int_{X}|u|_{\omega,h}^2 e^{\psi-\varphi} dV_{\omega}\leq&\liminf \limits_{k\to\infty}\int_X\chi_{X_k}|u|_{\omega,h}^2 e^{\psi-\varphi_k} dV_{\omega}
\\=&\liminf \limits_{k\to\infty}\int_{X_k} |u|_{\omega,h}^2 e^{\psi-\varphi_k} dV_{\omega}
\\\leq&\frac{1}{(1-\sqrt{r})^2}\int_X \langle[\Theta,\Lambda_{\omega}]^{-1}v,v\rangle _{\omega,h}e^{\psi - \varphi}  dV_{\omega}
\end{split}
\end{equation}
Therefore we have completed the proof of Proposition \ref{Berndtsson}.
\end{proof}

Now we come to the proof of Theorem\ref{main3}.
\begin{proof}(of Theorem\ref{main3})
Because $X$ is hyperconvex, there exists a smooth strictly psh function $\phi<0$ on $X$ such that $\phi^{-1}(-\infty,c)$ are relatively compact in $X$ for any negative $c$.

We take $\varphi=\psi=-\log(-\phi)$ and $r=\frac{1}{2}$, $\Theta:=i\partial\psi\wedge \bar\partial\psi$ in Proposition \ref{Berndtsson}. Then we have
\begin{equation}
\begin{split}    ri\partial\bar\partial(\varphi+\psi)=&i\partial\bar\partial\varphi
\\=&-\frac{i\partial\bar\partial\phi}{\phi}+\frac{i\partial\phi\wedge\bar\partial\phi}{\phi^2}
\\\geq&i\partial\psi\wedge \bar\partial\psi=\Theta
\end{split}
\end{equation}

Given any $F\in A^2(X,K_X\otimes E,h)$, we take $$v_j=\bar\partial((\chi_j\circ\varphi)F)=(\chi_j'\circ\varphi)\bar\partial\varphi\wedge F$$   
where $\chi_j$ is a smooth function on $\mathbb{R}$ such that $\chi_j(t)=1$ for $t\leq j$, $\chi_j(t)=0$ for $t\geq j+1$ and $\sup|\chi_j'|\leq 2$. 

Then $v_j$ is a smooth $E$-valued $\bar\partial$-closed $(n,1)$-form on $X$. By the equality (\ref{37}), we also have for any $j,k$
$$\langle[\Theta,\Lambda_{\omega}]^{-1}v_j,v_j\rangle _{\omega,h_k}=|\chi_j'\circ\varphi|^2|F|^2_{\omega,h_k}\leq 4|F|^2_{\omega,h_k}$$
Therefore by Proposition\ref{Berndtsson}, for any $j,k\in\mathbb{Z}_+$, there exists an $E$-valued $(n,0)$-form $u_{j,k}$ on $X$ such that $\bar\partial u_{j,k}=v_j$ and
\begin{equation}
\begin{split}
\int_X|u_{j,k}|^2_{\omega,h_k}dV_{\omega}\leq&\frac{1}{(1-\sqrt{\frac{1}{2}})^2}\int_X\langle[\Theta,\Lambda_{\omega}]^{-1}v_j,v_j\rangle_{\omega,h_k}dV_{\omega}
\\\leq&\frac{4}{(1-\sqrt{\frac{1}{2}})^2}\int_{\{j<\varphi<j+1\}}|F|^2_{\omega,h_k}dV_{\omega}
\end{split}
\end{equation}
We define 
$$F_{j,k}=(\chi_j\circ\varphi)F-u_{j,k}$$
By the assumption (2) of Theorem\ref{main3}, we know that $h_j\geq h$ and $\det h_j$ converges to $\det h$ almost everywhere. Hence by Theorem\ref{strong open v.b.}, $|F|^2_{\omega,h_k}$ converges to $|F|^2_{\omega,h}$ in $L^1_{loc}$ sense.

Hence for any $\epsilon>0$, there exists $j\in\mathbb{Z}_+$ such that 
\begin{equation}\label{42}
\int_{\{\varphi>j\}}|F|^2_{\omega,h}dV_{\omega}<\epsilon
\end{equation}
and $k_j\in\mathbb{Z}_+$ such that
\begin{equation}\label{43}
\int_{\{\varphi<j+1\}}|F|^2_{\omega,h_{k_j}}dV_{\omega}<\infty
\end{equation}
and
\begin{equation}\label{44}
\int_X|u_{j,k_j}|^2_{\omega,h_{k_j}}dV_{\omega}\leq\frac{4}{(1-\sqrt{\frac{1}{2}})^2}\int_{\{j<\varphi<j+1\}}|F|^2_{\omega,h_{k_j}}dV_{\omega}
<\frac{4}{(1-\sqrt{\frac{1}{2}})^2}\epsilon
\end{equation}
By the inequalities (\ref{42}), (\ref{43}) and (\ref{44}), we know that
$$F_{j,k_j}\in A^2(X,K_X\otimes E,h_{k_j})$$
and
$$\int_X|F_{j,k_j}-F|^2_{\omega,h}dV_{\omega}\leq 2(\int_{\{\varphi>j\}}|F|^2_{\omega,h}dV_{\omega}+\int_X|u_{j,k_j}|^2_{\omega,h}dV_{\omega})<2(1+\frac{4}{(1-\sqrt{\frac{1}{2}})^2})\epsilon$$
Therefore we have proved that $\cup_{j=1}^{\infty}A^2(X,K_X\otimes E,h_j)$is dense in $A^2(X,K_X\otimes E,h)$.
\end{proof}

\end{document}